\documentclass[oneside,10pt]{amsart}
\usepackage{amsthm,amssymb}
\usepackage{mathtools}
\usepackage[hidelinks]{hyperref}
\usepackage[a4paper,width=160mm,top=27mm,bottom=27mm]{geometry}
\numberwithin{equation}{section}
\usepackage{todonotes}
\usepackage{color}

\newtheorem{theorem}{Theorem}[section]
\newtheorem{lemma}[theorem]{Lemma}

\theoremstyle{definition}

\begin{document}
\address{Yongming Luo
\newline \indent Institut f\"ur Wissenschaftliches Rechnen, Technische Universit\"at Dresden\indent
\newline \indent  Zellescher Weg 25, 01069 Dresden, Germany.\indent }
\email{yongming.luo@tu-dresden.de}

\newcommand{\diver}{\operatorname{div}}
\newcommand{\lin}{\operatorname{Lin}}
\newcommand{\curl}{\operatorname{curl}}
\newcommand{\ran}{\operatorname{Ran}}
\newcommand{\kernel}{\operatorname{Ker}}
\newcommand{\la}{\langle}
\newcommand{\ra}{\rangle}
\newcommand{\N}{\mathbb{N}}
\newcommand{\R}{\mathbb{R}}
\newcommand{\C}{\mathbb{C}}
\newcommand{\T}{\mathbb{T}}

%%%%%%%%%%%%%%%%%%%%%%%%%%%%%%%%%%%%%%%%
\newcommand{\ld}{\lambda}
\newcommand{\fai}{\varphi}
\newcommand{\0}{0}
\newcommand{\n}{\mathbf{n}}
\newcommand{\uu}{{\boldsymbol{\mathrm{u}}}}
\newcommand{\UU}{{\boldsymbol{\mathrm{U}}}}
\newcommand{\buu}{\bar{{\boldsymbol{\mathrm{u}}}}}
\newcommand{\ten}{\\[4pt]}
\newcommand{\six}{\\[-3pt]}
\newcommand{\nb}{\nonumber}
\newcommand{\hgamma}{H_{\Gamma}^1(\OO)}
\newcommand{\opert}{O_{\varepsilon,h}}
\newcommand{\barx}{\bar{x}}
\newcommand{\barf}{\bar{f}}
\newcommand{\hatf}{\hat{f}}
\newcommand{\xoneeps}{x_1^{\varepsilon}}
\newcommand{\xh}{x_h}
\newcommand{\scaled}{\nabla_{1,h}}
\newcommand{\scaledb}{\widehat{\nabla}_{1,\gamma}}
\newcommand{\vare}{\varepsilon}
\newcommand{\A}{{\bf{A}}}
\newcommand{\RR}{{\bf{R}}}
\newcommand{\B}{{\bf{B}}}
\newcommand{\CC}{{\bf{C}}}
\newcommand{\D}{{\bf{D}}}
\newcommand{\K}{{\bf{K}}}
\newcommand{\oo}{{\bf{o}}}
\newcommand{\id}{{\bf{Id}}}
\newcommand{\E}{\mathcal{E}}
\newcommand{\ii}{\mathcal{I}}
\newcommand{\sym}{\mathrm{sym}}
\newcommand{\lt}{\left}
\newcommand{\rt}{\right}
\newcommand{\ro}{{\bf{r}}}
\newcommand{\so}{{\bf{s}}}
\newcommand{\e}{{\bf{e}}}
\newcommand{\ww}{{\boldsymbol{\mathrm{w}}}}
\newcommand{\zz}{{\boldsymbol{\mathrm{z}}}}
\newcommand{\U}{{\boldsymbol{\mathrm{U}}}}
\newcommand{\G}{{\boldsymbol{\mathrm{G}}}}
\newcommand{\VV}{{\boldsymbol{\mathrm{V}}}}
\newcommand{\II}{{\boldsymbol{\mathrm{I}}}}
\newcommand{\ZZ}{{\boldsymbol{\mathrm{Z}}}}
\newcommand{\hKK}{{{\bf{K}}}}
\newcommand{\f}{{\bf{f}}}
\newcommand{\g}{{\bf{g}}}
\newcommand{\lkk}{{\bf{k}}}
\newcommand{\tkk}{{\tilde{\bf{k}}}}
\newcommand{\W}{{\boldsymbol{\mathrm{W}}}}
\newcommand{\Y}{{\boldsymbol{\mathrm{Y}}}}
\newcommand{\EE}{{\boldsymbol{\mathrm{E}}}}
\newcommand{\F}{{\bf{F}}}
\newcommand{\spacev}{\mathcal{V}}
\newcommand{\spacevg}{\mathcal{V}^{\gamma}(\Omega\times S)}
\newcommand{\spacevb}{\bar{\mathcal{V}}^{\gamma}(\Omega\times S)}
\newcommand{\spaces}{\mathcal{S}}
\newcommand{\spacesg}{\mathcal{S}^{\gamma}(\Omega\times S)}
\newcommand{\spacesb}{\bar{\mathcal{S}}^{\gamma}(\Omega\times S)}
\newcommand{\skews}{H^1_{\barx,\mathrm{skew}}}
\newcommand{\kk}{\mathcal{K}}
\newcommand{\OO}{O}
\newcommand{\bhe}{{\bf{B}}_{\vare,h}}
\newcommand{\pp}{{\mathbb{P}}}
\newcommand{\ff}{{\mathcal{F}}}
\newcommand{\mWk}{{\mathcal{W}}^{k,2}(\Omega)}
\newcommand{\mWa}{{\mathcal{W}}^{1,2}(\Omega)}
\newcommand{\mWb}{{\mathcal{W}}^{2,2}(\Omega)}
\newcommand{\twos}{\xrightharpoonup{2}}
\newcommand{\twoss}{\xrightarrow{2}}
\newcommand{\bw}{\bar{w}}
\newcommand{\bz}{\bar{{\bf{z}}}}
\newcommand{\tw}{{W}}
\newcommand{\tr}{{{\bf{R}}}}
\newcommand{\tz}{{{\bf{Z}}}}
\newcommand{\lo}{{{\bf{o}}}}
\newcommand{\hoo}{H^1_{00}(0,L)}
\newcommand{\ho}{H^1_{0}(0,L)}
\newcommand{\hotwo}{H^1_{0}(0,L;\R^2)}
\newcommand{\hooo}{H^1_{00}(0,L;\R^2)}
\newcommand{\hhooo}{H^1_{00}(0,1;\R^2)}
\newcommand{\dsp}{d_{S}^{\bot}(\barx)}
\newcommand{\LB}{{\bf{\Lambda}}}
\newcommand{\LL}{\mathbb{L}}
\newcommand{\mL}{\mathcal{L}}
\newcommand{\mhL}{\widehat{\mathcal{L}}}
\newcommand{\loc}{\mathrm{loc}}
\newcommand{\tqq}{\mathcal{Q}^{*}}
\newcommand{\tii}{\mathcal{I}^{*}}
\newcommand{\Mts}{\mathbb{M}}
\newcommand{\pot}{\mathrm{pot}}
\newcommand{\tU}{{\widehat{\bf{U}}}}
\newcommand{\tVV}{{\widehat{\bf{V}}}}
\newcommand{\pt}{\partial}
\newcommand{\bg}{\Big}
\newcommand{\hA}{\widehat{{\bf{A}}}}
\newcommand{\hB}{\widehat{{\bf{B}}}}
\newcommand{\hCC}{\widehat{{\bf{C}}}}
\newcommand{\hD}{\widehat{{\bf{D}}}}
\newcommand{\fder}{\partial^{\mathrm{MD}}}
\newcommand{\Var}{\mathrm{Var}}
\newcommand{\pta}{\partial^{0\bot}}
\newcommand{\ptaj}{(\partial^{0\bot})^*}
\newcommand{\ptb}{\partial^{1\bot}}
\newcommand{\ptbj}{(\partial^{1\bot})^*}
\newcommand{\geg}{\Lambda_\vare}
\newcommand{\tpta}{\tilde{\partial}^{0\bot}}
\newcommand{\tptb}{\tilde{\partial}^{1\bot}}
\newcommand{\ua}{u_\alpha}
\newcommand{\pa}{p\alpha}
\newcommand{\qa}{q(1-\alpha)}
\newcommand{\Qa}{Q_\alpha}
\newcommand{\Qb}{Q_\eta}
\newcommand{\ga}{\gamma_\alpha}
\newcommand{\gb}{\gamma_\eta}
\newcommand{\ta}{\theta_\alpha}
\newcommand{\tb}{\theta_\eta}

%%%%%%%%%%%%%%%%%%%%%%%%%%

\newcommand{\mH}{{E}}
\newcommand{\mN}{{N}}
\newcommand{\mD}{{\mathcal{D}}}
\newcommand{\csob}{\mathcal{S}}
\newcommand{\mA}{{A}}
\newcommand{\mK}{{Q}}
\newcommand{\mS}{{S}}
\newcommand{\mI}{{I}}
\newcommand{\tas}{{2_*}}
\newcommand{\tbs}{{2^*}}
\newcommand{\tm}{{\tilde{m}}}
\newcommand{\tdu}{{\phi}}
\newcommand{\tpsi}{{\tilde{\psi}}}
\newcommand{\Z}{{\mathbb{Z}}}
\newcommand{\tsigma}{{\tilde{\sigma}}}
\newcommand{\tg}{{\tilde{g}}}
\newcommand{\tG}{{\tilde{G}}}
\newcommand{\mM}{{M}}
\newcommand{\mC}{\mathcal{C}}
\newcommand{\wlim}{{\text{w-lim}}\,}
\newcommand{\diag}{L_t^\ba L_x^\br}
\newcommand{\vu}{ u}
\newcommand{\vz}{ z}
\newcommand{\vv}{ v}
\newcommand{\ve}{ e}
\newcommand{\vw}{ w}
\newcommand{\vf}{ f}
\newcommand{\vh}{ h}
\newcommand{\vp}{ \vec P}
\newcommand{\ang}{{\not\negmedspace\nabla}}
\newcommand{\dxy}{\Delta_{x,y}}
\newcommand{\lxy}{L_{x,y}}
\newcommand{\gnsand}{\mathrm{C}_{\mathrm{GN},3d}}
\newcommand{\wmM}{\widehat{{M}}}
\newcommand{\wmH}{\widehat{{E}}}
\newcommand{\wmI}{\widehat{{I}}}
\newcommand{\wmK}{\widehat{{Q}}}
\newcommand{\wmN}{\widehat{{N}}}
\newcommand{\wm}{\widehat{m}}
\newcommand{\ba}{\mathbf{a}}
\newcommand{\bb}{\mathbf{b}}
\newcommand{\br}{\mathbf{r}}
\newcommand{\bs}{\mathbf{s}}
\newcommand{\bq}{\mathbf{q}}
\newcommand{\SSS}{\mathcal{S}}
%-------------------------------------

%New Commands

%-------------------------------------

%%%%%%%%%%%%%%%%%%%%%%
\title[Sharp scattering for focusing intercritical NLS on $\R^d\times \T$]{Sharp scattering for focusing intercritical NLS on high-dimensional waveguide manifolds}
\author{Yongming Luo}

%\date{}
\maketitle

\begin{abstract}
We study the focusing intercritical NLS
\begin{align}\label{abstract_nls}
i\pt_t u+\Delta_{x,y}u=-|u|^\alpha u\tag{NLS}
\end{align}
on the semiperiodic waveguide manifold $\R^d_x\times \T_y$ with $d\geq 5$ and $\alpha\in(\frac{4}{d},\frac{4}{d-1})$. In the case $d\leq 4$, with the aid of the semivirial vanishing theory \cite{Luo_inter}, the author was able to construct a sharp threshold, which being uniquely characterized by the ground state solutions, that sharply determines the bifurcation of global scattering and finite time blow-up solutions in dependence of the sign of the semivirial functional. As the derivative of the nonlinear potential is no longer Lipschitz in $d\geq 5$ and the underlying domain possesses an anisotropic nature, the proof in \cite{Luo_inter}, which makes use of the concentration compactness principle, can not be extended to higher dimensional models. In this paper, we exploit a well-tailored adaptation of the interaction Morawetz-Dodson-Murphy (IMDM) estimates, which were only known to be applicable on Euclidean spaces, into the waveguide setting, in order to prove that the large data scattering result formulated in \cite{Luo_inter} continues to hold for all $d\geq 5$. Together with Tzvetkov-Visciglia \cite{TzvetkovVisciglia2016} and the author \cite{Luo_inter}, we thus give a complete characterization of the large data scattering for \eqref{abstract_nls} in both defocusing and focusing case and in arbitrary dimension.
%It is also worth noting that in most studies, a proof based on the IMDM-estimates usually gives a shorter and simpler proof for the statements which were originally shown using the concentration compactness principle. Our paper hence gives the first application of the IMDM-estimates where the concentration compactness principle is so far unknown to be applicable.
\end{abstract}

\section{Introduction}
The paper is devoted to the study of large data scattering for the focusing nonlinear Schr\"odinger equation (NLS)
\begin{align}\label{nls}
i\pt_t u+\Delta_{x,y} u=-|u|^\alpha u
\end{align}
on high-dimensional waveguide manifolds $\R_x^d\times \T_y$ with $d\geq 5$, $\T=\R/2\pi\Z$ and $\alpha$ in the intercritical regime $(\frac{4}{d},\frac{4}{d-1})$. The NLS \eqref{nls} serves as a toy model in numerous physical applications such as nonlinear optics and Bose-Einstein condensation \cite{waveguide_ref_1,waveguide_ref_2,waveguide_ref_3}. One of the most interesting features displayed by \eqref{nls} is its semiperiodicity, corresponding to a partial confinement forcing the wave function to periodically move along some given directions. This particularly leads to some unexpected new challenges for the mathematical analysis. Different results concerning the Cauchy problem, long time behaviors and ground state solutions of \eqref{nls} have been well-established. In this direction, we refer for instance to the papers \cite{TNCommPDE,TzvetkovVisciglia2016,TTVproduct2014,Ionescu2,HaniPausader,CubicR2T1Scattering,R1T1Scattering,Cheng_JMAA,ZhaoZheng2021,RmT1,YuYueZhao2021,Luo_Waveguide_MassCritical,Luo_inter,Luo_energy_crit}.

Among all, we underline that the defocuisng analogue of \eqref{nls} was first studied by Tzvetkov and Visciglia \cite{TzvetkovVisciglia2016}. Therein, by making use of the Strichartz estimates of mixed type derived in \cite{TNCommPDE}, the authors proved that the defocusing \eqref{nls} is globally well posed for arbitrary initial data from $H^1(\R^d\times\T)$. Additionally, one of the novelties of \cite{TNCommPDE} is that by appealing to a suitable interaction Morawetz inequality, the authors were able to prove that a global solution of the defocusing \eqref{nls} scatters in time. More precisely, we have the following result:
\begin{theorem}[Large data scattering for defocusing \eqref{nls}, \cite{TNCommPDE}]\label{thm defocusing}
Let $u$ be a solution of the defocusing \eqref{nls}. Then there exist $\phi^{\pm}\in H^1(\R^d\times \T)$ such that
\begin{align}
\lim_{t\to\pm\infty}\|u(t)-e^{it\Delta_{x,y}}\phi^\pm\|_{H^1(\R^d\times\T)}=0.
\end{align}
\end{theorem}

Theorem \ref{thm defocusing} does not hold for the focusing \eqref{nls}. To see this, we may simply assume that \eqref{nls} is independent of $y$ and in this case, \eqref{nls} reduces to the mass-(super)critical NLS on $\R^d$, which is known to have finite time blow-up solutions. In general, however, on the Euclidean space $\R^d$ it was proved \cite{KenigMerle2006,HolmerRadial,non_radial,focusing_sub_2011} that one is able to construct some sharp threshold, which is described by the positive and smooth ground state solutions $P\in H^1(\R^d)$ satisfying
\begin{align}\label{ground state eq}
-\Delta P+\beta P=P^{\alpha+1}
\end{align}
for some $\beta>0$, that sharply determines the bifurcation of global scattering and finite time blow-up solutions, in dependence of the sign of the virial functional
\begin{align}
\wmK(u):=\|\nabla_x u\|_{L^2(\R^d)}^2-\frac{\alpha d}{2(\alpha+2)}\|u\|_{L^{\alpha+2}(\R^d)}^{\alpha+2}.
\end{align}
Inspired by the results from \cite{KenigMerle2006,HolmerRadial,non_radial,focusing_sub_2011}, we have proved in a previous paper \cite{Luo_inter} that similar large data scattering results also hold for \eqref{nls} on $\R^d\times\T$ with $d\leq 4$. The situation nevertheless becomes more subtle, since $\T$ is a compact manifold and therefore solutions of \eqref{nls} do not disperse along the $y$-direction. Therefore, to get a proper large data scattering result we shall appeal to the so-called \textit{semivirial-vanishing} framework. To introduce the theory we first fix some notations. Let $\mM(u)$ and $\mH(u)$ be the mass and energy of a solution $u$ of \eqref{nls} defined by \eqref{def of mass} and \eqref{def of mhu} respectively. We also define the semivirial functional $\mK(u)$ by
$$ \mK(u):=\|\nabla_x u\|_{L^2(\R^d\times\T)}^2-\frac{\alpha d}{2(\alpha+2)}\|u\|_{L^{\alpha+2}(\R^d\times\T)}^{\alpha+2}.$$
Notice that the coefficients of the semivirial $\mK(u)$ coincide with the ones of $\wmK(u)$, i.e. $\mK(u)$ is simply the integration of $\wmK(u)$ over $\T$. Next, for $c\in (0,\infty)$ we define
\begin{gather*}
S(c):=\{u\in H^1(\R^d\times\T):\mM(u)=c\},\\
V(c):=\{u\in S(c):\mK(u)=0\},\\
m_c:=\inf\{\mH(u):u\in V(c)\}.
\end{gather*}
By combining ideas from \cite{Akahori2013,LeCoz2008,Bellazzini2013} and a non-existence result \cite{Liouville_type} for the zero-mass problem
$$-\Delta_{x,y}u=u^{\alpha+1}$$
we were able to prove the following existence result for the ground state equation \eqref{ground state eq}.
\begin{theorem}[Existence of ground states, \cite{Luo_inter}]\label{thm existence of ground state}
For any $c\in(0,\infty)$ the minimization problem $m_c$ has a positive optimizer $u_c$. Moreover, $u_c$ solves the ground state equation \eqref{ground state eq} on $\R^d\times\T$ with some $\beta=\beta_c>0$.
\end{theorem}

Notice that by the boundedness of $\T$, we see that a solution of \eqref{ground state eq} on $\R^d$ is automatically a solution of \eqref{ground state eq} on $\R^d\times\T$. We therefore naturally ask whether the ground states deduced from Theorem \ref{thm existence of ground state} coincide with the ones on $\R^d$, or in other words whether they are $y$-independent. Following the crucial scaling arguments from \cite{TTVproduct2014} we indeed proved the following $y$-dependence result.

\begin{theorem}[$y$-dependence of ground states, \cite{Luo_inter}]\label{thm threshold mass}
Let $\wmM(u)$ and $\wmH(u)$ be the mass and energy of a function $u\in H^1(\R^d)$ defined on $\R^d$. Define also
\begin{gather*}
\widehat{S}(c):=\{u\in H^1(\R^d):\wmM(u)=c\},\\
\widehat{V}(c):=\{u\in \widehat{S}(c):\wmK(u)=0\},\\
\widehat{m}_c:=\inf\{\wmH(u):u\in \widehat{V}(c)\}.
\end{gather*}
Then there exists some $c_*\in(0,\infty)$ such that
\begin{itemize}
\item For all $c \in (0,c_*)$ we have $m_{c}<2\pi \wm_{(2\pi)^{-1}c}$. Moreover, for $c\in(0,c_*)$ any minimizer $u_c$ of $m_{c}$ satisfies $\pt_y u_c\neq 0$.

\item For all $c\in[c_*,\infty)$ we have $m_{c}=2\pi \wm_{(2\pi)^{-1}c}$. Moreover, for $c\in(c_*,\infty)$ any minimizer $u_c$ of $m_{c}$ satisfies $\pt_y u_c=0$.
\end{itemize}
\end{theorem}

Finally, using the concentration compactness principle and the Glassey's virial identity we deduced in \cite{Luo_inter} the following large data scattering and blow-up results.

\begin{theorem}[Scattering below ground states, \cite{Luo_inter}]\label{thm scattering d4}
Let $u$ be a solution of \eqref{nls}. If $\mH(u)<m_{\mM(u)}$ and $\mK(u(0))>0$, then $u$ is a global solution of \eqref{nls}. If additionally $d\leq 4$, then $u$ also scatters in time.
\end{theorem}

\begin{theorem}[Finite time blow-up below ground states, \cite{Luo_inter}]\label{thm blow up}
Let $u$ be a solution of \eqref{nls}. If $|x|u(0)\in L^2(\R^d\times\T)$, $\mH(u)<m_{\mM(u)}$ and $\mK(u(0))<0$, then $u$ blows-up in finite time.
\end{theorem}

We note that the scattering result in Theorem \ref{thm scattering d4} only holds in $d\leq 4$, which is due to the fact that the derivative of $|z|^\alpha z$ is no longer Lipschitz in $d\geq 5$. Indeed, this issue already arises when one studies the large data scattering problem for NLS on high-dimensional Euclidean spaces. To overcome this difficulty, one may appeal to the fractional calculus for proving a weaker stability result, which already suffices for many applications. We refer to \cite{defocusing5dandhigher} for more details in this direction. However, as we shall see in the following, the anisotropic nature of $\R^d\times\T$ forces us to take different orders of fractional derivatives w.r.t. the $x$- and $y$-directions, and so far it is unknown how to prove a suitable stability result in the waveguide setting. This prevented us to extend Theorem \ref{thm scattering d4} to higher dimensions $d\geq 5$.

With the aid of the interaction Morawetz-Dodson-Murphy (IMDM) estimates we shall prove the following main theorem:
\begin{theorem}\label{main result}
The large data scattering result formulated in Theorem \ref{thm scattering d4} continues to hold for all $d\geq 5$.
\end{theorem}

Here follow several comments on the IMDM-estimates and Theorem \ref{main result}. Originally, the interaction Morawetz estimate was first introduced in \cite{defocusing3d} for proving the large data scattering of the defocusing quintic NLS on $\R^3$ and further applied in \cite{GinibreVelo2010,2dinteraction,Zhang2006,TaoVisanZhang} for deriving long time dynamics results for defocusing NLS on $\R^d$. In the waveguide setting, this was first used by Tzvetkov and Visciglia \cite{TzvetkovVisciglia2016} to show Theorem \ref{thm defocusing}. For focusing problems, Dodson and Murphy \cite{DodsonMurphyRadial,DodsonMurphyNonRadial} made use of some well-designed Morawetz potentials, following the same fashion as the ones from \cite{Dodson4dmassfocusing,Dodson4dfocusing}, to prove some uniform space-time bounds for $\dot{H}^{\frac{1}{2}}$-critical focusing NLS on $\R^d$, from which the large data scattering below ground states already follows. In comparison to the concentration compactness principle, the IMDM-estimates have the advantages that they provide much shorter proofs than the ones based on the concentration compactness principle and the uniform space-time bounds are given in Lebesgue norms, thus no fractional calculus is involved. We will therefore prove Theorem \ref{main result} by using the same IMDM-estimates deduced in \cite{DodsonMurphyNonRadial}. As we shall see, however, the proof in the waveguide setting is more involved and technical, since the scattering norm is involved with some fractional power $\pt_y^s$ w.r.t. the $y$-direction. We will overcome this difficulty by appealing to some well-tailored interpolation inequalities. It is also worth noting that in previous studies, the IMDM-estimates generally provided an alternative simpler and shorter proof for the large data scattering results which were originally proved using the concentration compactness principle. Our paper hence gives the first application of the IMDM-estimates where a proof based on the concentration compactness principle is so far unknown to be applicable.

For the readers' interest we also refer to the recent papers \cite{DingMorawetz,InteractionCombined} for further applications of the IMDM-estimates on focusing intercritical NLS and NLS with combined powers. We shall also make use of several ideas from these papers for our model.

\subsection{Notation and definitions}
We use the notation $A\lesssim B$ whenever there exists some positive constant $C$ such that $A\leq CB$. Similarly we define $A\gtrsim B$ and we use $A\sim B$ when $A\lesssim B\lesssim A$.

For simplicity, we ignore in most cases the dependence of the function spaces on their underlying domains and hide this dependence in their indices. For example $L_x^2=L^2(\R^d)$, $H_{x,y}^1= H^1(\R^d\times \T)$
and so on. However, when the space is involved with time, we still display the underlying temporal interval such as $L_t^pL_x^q(I)$, $L_t^\infty L_{x,y}^2(\R)$ etc. The norm $\|\cdot\|_p$ is defined by $\|\cdot\|_p:=\|\cdot\|_{L_{x,y}^p}$.

Next, we define the quantities such as mass and energy etc. that will be frequently used in the proof of the main results. For $u\in H_{x,y}^1$, define
\begin{align}
\mM(u)&:=\|u\|^2_{2},\label{def of mass}\\
\mH(u)&:=\frac{1}{2}\|\nabla_{x,y} u\|^2_{2}-\frac{1}{\alpha+2}\|u\|^{\alpha+2}_{\alpha+2},\label{def of mhu}\\
\mK(u)&:=\|\nabla_{x} u\|^2_{2}-\frac{\alpha d}{2(\alpha+2)}\|u\|^{\alpha+2}_{\alpha+2},\\
\mI(u)&:=\frac{1}{2}\|\pt_y u\|_{2}^2+\bg(\frac{1}{2}-\frac{2}{\alpha d}\bg)\|\nabla_{x}u\|_2^2=\mH(u)-\frac{2}{\alpha d}
\mK(u).\label{def of mI}
\end{align}
We also define the sets
\begin{align}
S(c)&:=\{u\in H_{x,y}^1:\mM(u)=c\},\\
V(c)&:=\{u\in S(c):\mK(u)=0\}
\end{align}
and the variational problem
\begin{align}
m_c&:=\inf\{\mH(u):u\in V(c)\}\label{def of mc}.
\end{align}
Finally, for a function $u\in H_{x,y}^1$, the scaling operator $u\mapsto u^t$ for $t\in(0,\infty)$ is defined by
\begin{align}\label{def of scaling op}
u^t(x,y):=t^{\frac d2}u(tx,y).
\end{align}

Next we introduce the concept of an \textit{admissible} pair on $\R^d$. A pair $(q,r)$ is said to be ${H}^s$-admissible for $s\in[0,\frac d2)$ if $q,r\in[2,\infty]$, $\frac{2}{q}+\frac{d}{r}=\frac{d}{2}-s$ and $(q,d)\neq(2,2)$. For any $L^2$-admissible pairs $(q_1,r_1)$ and $(q_2,r_2)$ we have the following Strichartz estimate: if $u$ is a solution of
\begin{align*}
i\pt_t u+\Delta_x u=F
\end{align*}
on $I\subset\R$ with $t_0\in I$ and $u(t_0)=u_0$, then
\begin{align}
\|u\|_{L_t^q L_x^r(I)}\lesssim \|u_0\|_{L_x^2}+\|F\|_{L_t^{q_2'} L_x^{r_2'}(I)},
\end{align}
where $(q_2',r_2')$ is the H\"older conjugate of $(q_2,r_2)$. For a proof, we refer to \cite{EndpointStrichartz,Cazenave2003}. For $d\geq 3$, we define the space $S_x$ by
\begin{align}
S_x:=L_t^\infty L_x^2\cap L^{2}_t L_x^{\frac{2d}{d-2}}\label{def of sx}.
\end{align}
For $d\in\{1,2\}$, the space $L^{2}_t L_x^{\frac{2d}{d-2}}$ in the definition of $S_x$ is replaced by
\begin{align*}
d=1:\quad L^{4}_t L_x^{\infty}\quad\text{and}\quad
d=2:\quad L^{2^+}_t L_x^{\infty^-}
\end{align*}
respectively, where $(2^{+},\infty^-)$ is an $L^2$-admissible pair with some $2^+\in(2,\infty)$ sufficiently close to $2$. In the following, an admissible pair is always referred to as an $L^2$-admissible pair if not otherwise specified.

\subsubsection*{The numbers $s_\alpha$ and $s$}
Throughout the paper we set $s_\alpha:=\frac{d}{2}-\frac{2}{\alpha}\in(0,\frac{1}{2})$ and $s$ to be some fixed number such that $s\in(\frac{1}{2},1-s_\alpha)$.

\section{Some useful inequalities and a local control result}
In this section we collect some useful tools which will be used throughout the rest of the paper.

\begin{lemma}[Strichartz estimates on $\R^d\times\T$, \cite{TzvetkovVisciglia2016}]\label{strichartz}
Let $\gamma\in\R$, $\kappa\in[0,d/2)$ and $p,q,\tilde{p},\tilde{q}$ satisfy $p,\tilde{p}\in(2,\infty]$ and
\begin{align}
\frac{2}{p}+\frac{d}{q}=\frac{2}{\tilde{p}}+\frac{d}{\tilde{q}}
=\frac{d}{2}-\kappa.
\end{align}
Then for a time interval $I\ni t_0$ we have
\begin{align}
\|e^{i(t-t_0)\dxy}f\|_{L_t^p L_x^q H^\gamma_y(I)}&\lesssim\|f\|_{{H}_x^\kappa H_y^\gamma}.
\end{align}
Moreover, the Strichartz estimate for the nonlinear term
\begin{align}
\|\int_{t_0}^t e^{i(t-s)\dxy}F(s)\,ds\|_{L_t^p L_x^q H^\gamma_y(I)}&\lesssim \|F\|_{L_t^{\tilde{p}'} L_x^{\tilde{q}'} H^\gamma_y(I)}
\end{align}
holds in the case $\kappa=0$.
\end{lemma}

\begin{lemma}[Exotic Strichartz estimates on $\R^d\times\T$, \cite{TzvetkovVisciglia2016}]\label{exotic strichartz}
There exist $\ba,\br,\bb,\bs\in(2,\infty)$ such that
\begin{gather*}
(\alpha+1)\bs'=\br,\quad(\alpha+1)\bb'=\ba,\quad\alpha/\br<\min\{1,\frac{2}{d}\},\quad
\frac{2}{\ba}+\frac{d}{\br}=\frac{2}{\alpha}.
\end{gather*}
Moreover, for any $\gamma\in\R$ we have the following exotic Strichartz estimate:
\begin{align}
\|\int_{t_0}^t e^{i(t-s)\dxy}F(s)\,ds\|_{L_t^\ba L_x^\br H^\gamma_y(I)}&\lesssim \|F\|_{L_t^{\bb'} L_x^{\bs'} H^\gamma_y(I)}.
\end{align}
When $d\geq 5$, we can additionally assume that there exists some $0<\beta\ll 1$ such that $\br$ can be chosen as an arbitrary number from $(\frac{\alpha(\alpha+1)d}{\alpha+2},\frac{\alpha(\alpha+1)d}{\alpha+2}+\beta)$.
\end{lemma}

\begin{lemma}[A useful H\"older identity]\label{lemma l2 admissible special}
Let $(\tilde{\ba},\tilde{\br})=(\frac{4\br}{\alpha d},\frac{2\br}{\br-\alpha})$. Then $(\tilde{\ba},\tilde{\br})$ is an $L^2$-admissible pair and
\begin{align}\label{l2 admissible exotic}
\frac{1}{\tilde{\ba}'}=\frac{\alpha}{\ba}+\frac{1}{\tilde{\ba}},\quad\frac{1}{\tilde{\br}'}=\frac{\alpha}{\br}+\frac{1}{\tilde{\br}}
\end{align}
is satisfied.
\end{lemma}

\begin{proof}
It is easy to see that \eqref{l2 admissible exotic} is satisfied for the given $(\ba,\br)$ and $\frac{2}{\ba}+\frac{d}{\br}=\frac{d}{2}$. It remains to show that $\br\in(2,2^*)$, where $2^*=\infty$ when $d=1,2$ and $2^*=\frac{2d}{d-2}$ when $d\geq 3$, which in turn implies that $(\ba,\br)$ is an $L^2$-admissible pair. We discuss different cases:
\begin{itemize}
\item For $d\in\{1,2\}$, we have $\tilde{\br}=\frac{2\br}{\br-\alpha}\in(2,\infty)$ and we obtain an admissible choice.
\item For $d\geq 3$, we may rewrite $\tilde{\br}=\frac{2\br}{\br-\alpha}$ to $\tilde{\br}=\frac{2(2\br/\alpha)}{(2\br/\alpha)-2}$. Notice that the function $l\mapsto\frac{2l}{l-2}$ is monotone decreasing on $(2,\infty)$. From Lemma \ref{exotic strichartz} we know that $2\ro/\alpha\in(d,\infty)$. Thus $\tilde{\br}\in(2,\frac{2d}{d-2})$ and we obtain an admissible choice.
\end{itemize}
\end{proof}

\begin{lemma}[Fractional calculus on $\T$, \cite{TzvetkovVisciglia2016}]\label{fractional on t}
For $s\in(0,1)$ and $\alpha>0$ we have
\begin{align}
\|u^{\alpha+1}\|_{\dot{H}_y^s}+\||u|^{\alpha+1}\|_{\dot{H}_y^s}+\||u|^\alpha u\|_{\dot{H}_y^s}\lesssim_{\alpha,s} \|u\|_{\dot{H}_y^s}\|u\|_{L_y^\infty}^\alpha
\end{align}
for $u\in H_y^s$.
\end{lemma}

We also record a useful local control result for a solution of \eqref{nls} at the end of the section.
\begin{lemma}[Local control for a solution of \eqref{nls}]\label{local control}
Let $u$ be a global solution of \eqref{nls} with $\|u\|_{L_t^\infty H_{x,y}^1(\R)}<\infty$ and let $s\in(\frac{1}{2},1-s_\alpha)$, where $s_\alpha=\frac{d}{2}-\frac{2}{\alpha}$. Then for any $L^2$-admissible pair $(p,q)$ with $p\neq \infty$ we have
\begin{align}
\|u\|_{L_t^p W_x^{1-s,q}H_y^s(I)}\lesssim \la I\ra^{\frac{1}{p}}.
\end{align}
\end{lemma}

\begin{proof}
Let $\br_s>\frac{\alpha d}{2}$ be a number sufficiently close to $\frac{\alpha d}{2}$ such that $H_x^{1-s}\hookrightarrow L_x^{\br_s}$. This is always possible since $1-s>s_\alpha$ and $H_x^{s_\alpha}\hookrightarrow L_x^{\frac{\alpha d}{2}}$. Let also $\ba_s$ be given such that $\frac{2}{\ba_s}+\frac{d}{\br_s}=\frac{2}{\alpha}$. Notice that
\begin{align}
\frac{1}{(2d/(d+2))}=\frac{\alpha}{(\alpha d/2)}+\frac{1}{(2d/(d-2))},\quad
\frac{1}{2}=\frac{\alpha}{\infty}+\frac{1}{2},\quad\frac{2}{\infty}+\frac{d}{(\alpha d/2)}=\frac{2}{\alpha}.
\end{align}
Hence when $\br_s$ is sufficiently close to $\frac{\alpha d}{2}$, we can find an $L^2$-admissible pair $(p_s,q_s)$ such that
\begin{align}
\frac{1}{(2d/(d+2))}=\frac{\alpha}{\br_s}+\frac{1}{q_s},\quad
\frac{1}{2}=\frac{\alpha}{\ba_s}+\frac{1}{p_s}.
\end{align}
Let
$$X:=\sum_{D\in\{1,\pt_y,\nabla_x\}}\|Du\|_{L_t^p L_x^q L_y^2(I)}+\sum_{D\in\{1,\pt_y,\nabla_x\}}\|Du\|_{L_t^{p_s} L_x^{q_s} L_y^2(I)}.$$
Using Strichartz, H\"older and the embedding $H_y^s\hookrightarrow L_y^\infty$ we obtain
\begin{align}
X&\lesssim 1+\sum_{D\in\{1,\pt_y,\nabla_x\}}\|D(|u|^\alpha u)\|_{L_t^{2}L_x^{\frac{2d}{d+2}} L_y^2(I)}\nonumber\\
&\lesssim 1+\sum_{D\in\{1,\pt_y,\nabla_x\}}\|\|u\|^\alpha_{L_x^{\br_s}H_y^s}\|Du\|_{L_x^{q_s}L_y^2}\|_{L_t^{2}(I)}\nonumber\\
&\lesssim 1+|I|^{\frac{\alpha}{\ba_s}}\|u\|^\alpha_{L_t^\infty L_x^{\br_s}H_y^s(I)}\bg(\sum_{D\in\{1,\pt_y,\nabla_x\}}\|Du\|_{L_t^{p_s} L_x^{q_s}L_y^2(I)}\bg)\nonumber\\
&\lesssim 1+|I|^{\frac{\alpha}{\ba_s}}\|u\|^\alpha_{L_t^\infty H_x^{1-s}H_y^s(I)}X
\lesssim 1+|I|^{\frac{\alpha}{\ba_s}}\|u\|^\alpha_{L_t^\infty H_{x,y}^{1}(I)}X\lesssim 1+|I|^{\frac{\alpha}{\ba_s}}X.
\end{align}
The desired claim then follows from a standard continuity argument and the interpolation
\begin{align}
\|f\|_{W_x^{1-s,r}H_y^s}\lesssim \|f\|_{L_x^r H_y^1}^s\|f\|_{W_x^{1,r} L_y^2}^{1-s}\label{amann inter}
\end{align}
for $r\in(1,\infty)$ (which can for instance be deduced using \cite[Thm. 3.1]{AmannEmbedding} by replacing the real interpolation therein to the complex one).
\end{proof}

\section{Scattering criterion}
In this section we give a scattering criterion for a solution of \eqref{nls}. As we shall see, such scattering criterion will follow from the IMDM-estimate deduced from Lemma \ref{lem inter morawetz} below.

\begin{lemma}[Scattering criterion]\label{lem scattering criterion}
Let $u$ be a global solution of \eqref{nls} and assume that $\|u\|_{L_t^\infty H_{x,y}^1(\R)}\leq A$. Then for any $\sigma>0$ there exist $\vare=\vare(\sigma,A)$ sufficiently small and $T_0=T_0(\sigma,\vare,A)$ sufficiently large such that if for all $a\in\R$ there exists $T\in(a,a+T_0)$ such that $[T-\vare^{-\sigma},T]\subset(a,a+T_0)$ and
\begin{align}
\|u\|_{L_t^\ba L_x^\br H_y^s(T-\vare^{-\sigma},T)}\lesssim \vare,\label{4.1}
\end{align}
then $u$ scatters forward in time.
\end{lemma}

\begin{proof}
By small data theory (see for instance \cite[Lem. 4.6]{Luo_inter}), it suffices to show that there exists some $T\gg 1$ such that
\begin{align}
\|e^{i(t-T)\Delta_{x,y}}u(T)\|_{L_t^\ba L_x^\br H_y^s(T,\infty)}\lesssim\vare^\mu
\end{align}
for some $\mu>0$. Using Duhamel's formula we have
$$e^{i(t-T)\Delta_{x,y}}u(T)=e^{it\Delta_{x,y}}u_0+i\int_0^T e^{i(t-s)\Delta_{x,y}}(|u|^\alpha u)(s)\,ds. $$
First, using Strichartz we obtain
\begin{align}
\|e^{it\Delta_{x,y}}u\|_{L_t^\ba L_x^\br H_y^s(\R)}\lesssim \|u_0\|_{H_x^{s_\alpha} H_y^s}\lesssim \|u_0\|_{H_{x,y}^1}<\infty.
\end{align}
Thus we can find some large $T$ such that
\begin{align}\|e^{it\Delta_{x,y}}u\|_{L_t^\ba L_x^\br H_y^s(T,\infty)}\lesssim \vare.\end{align}
Next, rewrite
\begin{align}
\int_0^T e^{i(t-s)\Delta_{x,y}}(|u|^\alpha u)(s)\,ds&=
\int_0^{T-\vare^{-\sigma}} e^{i(t-s)\Delta_{x,y}}(|u|^\alpha u)(s)\,ds
+\int_{T-\vare^{-\sigma}}^{T} e^{i(t-s)\Delta_{x,y}}(|u|^\alpha u)(s)\,ds\nonumber\\
&=:F_1(t)+F_2(t).
\end{align}
By Strichartz and Lemma \ref{fractional on t}, we have
\begin{align}
\|F_2\|_{L_t^\ba L_x^\br H_y^s(T,\infty)}\lesssim \||u|^\alpha u\|_{L_t^{\br'}L_x^{\bs'}H_y^s(T-\vare^{-\sigma},T)}
\lesssim \|u\|_{L_t^\ba L_x^\br H_y^s(T-\vare^{-\sigma},T)}^{\alpha+1}\lesssim \vare^{\alpha+1}.
\end{align}
We let $\br_1<\br$ be given such that $(\ba,\br_1)$ is an $L^2$-admissible pair. Let $\br_2>\br$ and $\theta\in(0,1)$ be some constants to be determined later such that
\begin{align}
\br^{-1}=\theta\br_1^{-1}+(1-\theta)\br_2^{-1}.
\end{align}
By interpolation we therefore have
\begin{align}
\|F_1\|_{L_t^\ba L_x^\br H_y^s(T,\infty)}\leq \|F_1\|_{L_t^\ba L_x^{\br_1} H_y^s(T,\infty)}^\theta\|F_1\|_{L_t^\ba L_x^{\br_2} H_y^s(T,\infty)}^{1-\theta}.
\end{align}
Notice that
$$ F_1(t)=e^{i(t-T+\vare^{-\sigma})\Delta_{x,y}}u(T-\vare^{-\sigma})-e^{it\Delta_{x,y}}u_0.$$
Thus by Strichartz and the uniform $H_{x,y}^1$-boundedness of $u$ we obtain
\begin{align}
\|F_1\|_{L_t^\ba L_x^{\br_1} H_y^s(T,\infty)}\lesssim 1.
\end{align}
We now construct a suitable $\br_2$. Notice first that by Sobolev embedding, $\br_2'(\alpha+1)\in\bg[2,\frac{\alpha d}{2}\bg]$ will in turn implies $H_x^{s_\alpha}\hookrightarrow L_x^{\br_2'(\alpha+1)}$. Combining with dispersive estimate, Minkowski and Lemma \ref{fractional on t} we deduce
\begin{align*}
&\,\|F_1(t)\|_{L_x^{\br_2}H_y^s}
\lesssim \|F_1(t)\|_{L_x^{\br_2}L_y^2}+\|\pt_y^s F_1(t)\|_{L_x^{\br_2}L_y^2}
\lesssim \|F_1(t)\|_{L_y^2 L_x^{\br_2}}+\|\pt_y^s F_1(t)\|_{L_y^2 L_x^{\br_2}}\nonumber\\
\lesssim&\, \bg\|\int_0^{T-\vare^{-\sigma}}(t-s)^{-\frac{d}{2}(1-\frac{2}{\br_2})}\|e^{i(t-s)\Delta_y}|u(s)|^\alpha u(s)\|_{L_x^{\br_2'}}\,ds\bg\|_{L_y^2}\nonumber\\
+&\,\bg\|\int_0^{T-\vare^{-\sigma}}(t-s)^{-\frac{d}{2}(1-\frac{2}{\br_2})}\|e^{i(t-s)\Delta_y}\pt_y^s(|u(s)|^\alpha u(s))\|_{L_x^{\br_2'}}\,ds\bg\|_{L_y^2}\nonumber\\
\lesssim&\, \int_0^{T-\vare^{-\sigma}}(t-s)^{-\frac{d}{2}(1-\frac{2}{\br_2})}\||u(s)|^\alpha u(s)\|_{L_x^{\br_2'}L_y^2}\,ds
+\int_0^{T-\vare^{-\sigma}}(t-s)^{-\frac{d}{2}(1-\frac{2}{\br_2})}\|\pt_y^s(|u(s)|^\alpha u(s))\|_{L_x^{\br_2'}L_y^2}\,ds\nonumber\\
\lesssim&\, \int_0^{T-\vare^{-\sigma}}(t-s)^{-\frac{d}{2}(1-\frac{2}{\br_2})}\|u(s)\|^{\alpha+1}_{L_x^{(\alpha+1)\br_2'}H_y^s}\,ds\nonumber\\
\lesssim&\, \int_0^{T-\vare^{-\sigma}}(t-s)^{-\frac{d}{2}(1-\frac{2}{\br_2})}\|u(s)\|^{\alpha+1}_{L_y^2L_x^{(\alpha+1)\br_2'}}\,ds
+\int_0^{T-\vare^{-\sigma}}(t-s)^{-\frac{d}{2}(1-\frac{2}{\br_2})}\|\pt_y^s u(s)\|^{\alpha+1}_{L_y^2L_x^{(\alpha+1)\br_2'}}\,ds\nonumber\\
\lesssim&\, \int_0^{T-\vare^{-\sigma}}(t-s)^{-\frac{d}{2}(1-\frac{2}{\br_2})}\|u(s)\|^{\alpha+1}_{H_y^s H_x^{s_\alpha}}\,ds
\lesssim \int_0^{T-\vare^{-\sigma}}(t-s)^{-\frac{d}{2}(1-\frac{2}{\br_2})}\|u(s)\|^{\alpha+1}_{H_{x,y}^1}\,ds\nonumber\\
\lesssim &\,(t-T+\vare^{-\sigma})^{-\frac{d}{2}(1-\frac{2}{\br_2})+1},
\end{align*}
provided that
\begin{align}\label{cond1}
(\alpha+1)\br_2'\in\bg[2,\frac{\alpha d}{2}\bg],\quad\frac{d}{2}\bg(1-\frac{2}{\br_2}\bg)-1>0.
\end{align}
This in turn implies
\begin{align}
\|F_1\|_{L_t^\ba L_x^{\br_2} H_y^s(T,\infty)}\lesssim \bg(\int_T^\infty (t-T+\vare^{-\sigma})^{-\bg(\frac{d}{2}\bg(1-\frac{2}{\br_2}\bg)-1\bg){\ba}}\bg)^{\frac{1}{\ba}}\lesssim
\vare^{\sigma\bg(\frac{d}{2}\bg(1-\frac{2}{\br_2}\bg)-1-\frac{1}{\ba}\bg)},
\end{align}
provided that
\begin{align}\label{cond2}
\frac{d}{2}\bg(1-\frac{2}{\br_2}\bg)-1-\frac{1}{\ba}>0.
\end{align}
Inspired by \cite{DingMorawetz} we set $\br_2=\frac{2}{1-\alpha}$. We show that by this choice of $\br_2$, \eqref{cond1}, \eqref{cond2} and $\br_2>\br$ are satisfied when $\br$ is sufficiently close to $\frac{\alpha(\alpha+1)d}{\alpha+2}$ (which is allowable by Lemma \ref{exotic strichartz}). One easily verifies that $(\alpha+1)\br_2'=2<\frac{\alpha d}{2}$, since $\alpha>4/d$. Next, we have
\begin{align}
\frac{d}{2}\bg(1-\frac{2}{\br_2}\bg)-1>\frac{d}{2}\bg(1-\frac{2}{\br_2}\bg)-1-\frac{1}{\ba}
=\frac{\alpha d}{2}-1-\frac{1}{\ba}>1-\frac{1}{\ba}>0,
\end{align}
since $\ba>2$. It remains to prove $\br_2>\br$. In fact, we only need to show that $\br_2>\frac{\alpha(\alpha+1)d}{\alpha+2}$ since we can choose $\br$ sufficiently close to $\frac{\alpha(\alpha+1)d}{\alpha+2}$. One easily verifies
\begin{align}
\br_2=\frac{2}{1-\alpha}>\frac{\alpha(\alpha+1)d}{\alpha+2}
\Leftrightarrow d\alpha^3+(2-d)\alpha +4>0.
\end{align}
Define $f(\alpha):=d\alpha^3+(2-d)\alpha +4$. Then calculating the derivative we obtain that $f$ is monotone decreasing on $[0,((3d)^{-1}(d-2))^{\frac{1}{2}}]$ and increasing on $[((3d)^{-1}(d-2))^{\frac{1}{2}},\infty)$. Define therefore $\alpha_{\min}:=((3d)^{-1}(d-2))^{\frac{1}{2}}$. By fundamental calculation we infer that
\begin{align*}
f(4/d)=\frac{8d+64}{d^2}>0,\quad f(4/(d-1))=\frac{64d+4(d-1)^2}{(d-1)^3}>0.
\end{align*}
Hence if $\alpha_{\min}\leq 4/d$ or $\alpha_{\min}\geq 4/(d-1)$ then we are done. Otherwise we have $\alpha_{\min}\in(4/d,4/(d-1))$. Noticing that $\alpha_{\min}$ is a monotone increasing function in $d$, while $4/(d-1)$ is monotone decreasing. Thus when $\alpha_{\min}<4/(d-1)$ it is necessary that $ d\leq 8$. However,
\begin{align}
((3d)^{-1}(d-2))^{\frac{1}{2}}>4/d\Rightarrow d>8
\end{align}
and hence the underlying case is absurd. The desired proof is therefore complete.
\end{proof}

\section{Variational analysis and IMDM-estimates}
This section is devoted to the variational analysis for the NLS \eqref{nls} and the IMDM-estimates based on the derived variational results. First, let $\chi\in C_0^\infty(\R^d;[0,1])$ be a radially symmetric and decreasing cut-off function such that $\chi(|x|)=1$ for $|x|\leq 1-\eta$ and $\chi(|x|)=0$ for $|x|\geq 1$, where $0<\eta\ll 1$ is some to be determined small positive number. For $R>0$, we also write $\chi_R(x):=\chi(x/R)$. Moreover, we define the set
\begin{align}
\mathcal{A}:=\{u\in S(c):\mH(u)<m_c,\,\mK(u)>0\}.
\end{align}

We begin with some useful results proved in \cite{Luo_inter}.

\begin{lemma}[An alternative characterization of $m_c$, \cite{Luo_inter}]\label{lem tilde mc}
Let
$$\tilde{m}_c:=\inf\{\mI(u):u\in S(c),\,\mK(u)\leq 0\}.$$
Then $m_c=\tilde{m}_c$ for all $c\in(0,\infty)$.
\end{lemma}

\begin{lemma}[Property of the mapping $t\mapsto \mK(u^t)$, \cite{Luo_inter}]\label{monotoneproperty}
Let $c>0$ and $u\in S(c)$. Then the following statements hold true:
\begin{enumerate}
\item[(i)] $\frac{\partial}{\partial t}\mH(u^t)=t^{-1} Q(u^t)$ for all $t>0$.
\item[(ii)] There exists some $t^*=t^*(u)>0$ such that $u^{t^*}\in V(c)$.
\item[(iii)] We have $t^*<1$ if and only if $\mK(u)<0$. Moreover, $t^*=1$ if and only if $\mK(u)=0$.
\item[(iv)] Following inequalities hold:
\begin{equation*}
Q(u^t) \left\{
\begin{array}{lr}
             >0, &t\in(0,t^*) ,\\
             <0, &t\in(t^*,\infty).
             \end{array}
\right.
\end{equation*}
\item[(v)] $\mH(u^t)<\mH(u^{t^*})$ for all $t>0$ with $t\neq t^*$.
\end{enumerate}
\end{lemma}

\begin{lemma}[Property of the mapping $c\mapsto m_c$, \cite{Luo_inter}]\label{monotone lemma}
The mapping $c\mapsto m_c$ is continuous and monotone decreasing on $(0,\infty)$.
\end{lemma}

We now follow the same lines as in \cite{InteractionCombined} to show the following coercivity result.

\begin{lemma}[A coercivity argument]\label{lem coercive}
Let $u$ be a solution of \eqref{nls} with $u(0)\in\mathcal{A}$. Then $u$ is global and $u(t)\in\mathcal{A}$ for all $t\in\R$. Moreover, there exist $0<\delta\ll 1$ and $R_0\gg 1$ such that for all $R\geq R_0$, $z\in \R^d$, $t\in\R$ we have
\begin{align}
\mK(\chi_R(\cdot-z)u^\xi(t))\geq \delta\|\nabla_x(\chi_R(\cdot-z))u^\xi(t)\|^2_{2},\label{5.2}
\end{align}
where $u^\xi(t,x,y):=e^{ix\cdot\xi}u(t,x,y)$ and
\begin{equation*}
\xi=\xi(t,z,R)=\left\{
\begin{array}{rr}
             -\frac{\int \mathrm{Im}(\chi_R^2(x-z)\bar{u}(t,x,y)\nabla_x u(t,x,y))\,dxdy}{\int \chi_R^2(x-z)|u(t,x,y)|^2\,dxdy}, &
             \text{if }\int \chi_R^2(x-z)|u(t,x,y)|^2\,dxdy\neq 0,\\
             0, &\text{if }\int \chi_R^2(x-z)|u(t,x,y)|^2\,dxdy=0.
             \end{array}
\right.
\end{equation*}
\end{lemma}

\begin{proof}
By \cite[Lem. 4.15, Lem. 4.17]{Luo_inter} we know that $u(t)\in \mathcal{A}$ for all $t\in I_{\max}$ and $\sup_{t\in I_{\max}}\|u(t)\|_{H_{x,y}^1}<\infty$, where $I_{\max}$ is the maximal lifespan of $u$. Since \eqref{nls} is energy-subcritical, the uniform $H_{x,y}^1$-boundedness of $u$ implies that $u$ is global.

Next we prove \eqref{5.2}. Assume $\mH(u)=m_c-\nu$ for some $\nu>0$.
%By the definition of $\mI$ we have
%$$ \mI(u(t))=\mH(u(t))-\frac{2}{\alpha d}\mK(u(t))\leq \mH(u(t))\leq \m_c-\nu.$$
Using the product rules
\begin{align*}
\int|\nabla_{x}(\chi u)|^2\,dxdy&=\int\chi^2|\nabla_x u|^2\,dxdy-\int \chi \Delta_x \chi |u|^2\,dxdy,\\
\int|\pt_y(\chi u)|^2\,dxdy&=\int\chi^2|\pt_y u|^2\,dxdy
\end{align*}
we obtain
\begin{align*}
\int|\nabla_{x}(\chi u^\xi)|^2\,dxdy&=|\xi|^2\int\chi^2|u|^2\,dxdy+\int\chi^2|\nabla_xu|^2\,dxdy\nonumber\\
&-\int \chi \Delta_x \chi |u|^2\,dxdy+2\xi\cdot\int\mathrm{Im}(\chi^2 \bar{u}\nabla_x u)\,dxdy,\\
\int|\pt_y(\chi u^\xi)|^2\,dxdy&=\int\chi^2|\pt_y u|^2\,dxdy.
\end{align*}
Hence if $\int \chi_R^2(x-z)|u(t,x,y)|^2\,dxdy\neq 0$, we have
\begin{align}
\mI(\chi_R(x-z) u^\xi(t))&=\frac{1}{2}\int \chi_R^2(x-z)|\pt_y u(t)|^2\,dxdy+\bg(\frac{1}{2}-\frac{2}{\alpha d}\bg)\int \chi_R^2(x-z)|\pt_{x} u(t)|^2\,dxdy\nonumber\\
&-\frac{\bg(\int \mathrm{Im}(\chi_R^2(x-z)\bar{u}(t,x,y)\nabla_x u(t,x,y))\,dxdy\bg)^2}{\int \chi_R^2(x-z)|u(t,x,y)|^2\,dxdy}
\int \chi_R^2(x-z)|u(t)|^2\,dxdy\nonumber\\
&-\int \chi_R(x-z)\Delta_x(\chi_R(x-z))|u(t)|^2\,dxdy\nonumber\\
&\leq \mI(u(t))+O(R^{-2}).\label{5.3}
\end{align}
In the case $\int \chi_R^2(x-z)|u(t,x,y)|^2\,dxdy= 0$, we can similarly deduce \eqref{5.3}. Thus we choose $R_0\gg 1$ such that $O(R^{-2})\leq \frac{\nu}{2}$ for all $R\geq R_0$, which combining with $\mH(u)=m_c-\nu$ and $\mK(u(t))>0$ implies
\begin{align}
\mI(\chi_R(x-z)u^\xi(t))\leq \mI(u(t))+\frac{\nu}{2} =\mH(u(t))-\frac{2}{\alpha d}\mK(u(t))+\frac{\nu}{2}\leq m_c-\frac{\nu}{2}.\label{5.4}
\end{align}
Then by Lemma \ref{lem tilde mc} we must have $\mK(\chi_R(x-z)u^\xi(t))>0$.

In the following we set $\Theta:=\chi_R(x-z)u^\xi(t)$. We finish the remaining steps of the proof by discussing different cases.
\begin{itemize}
\item First we assume that
$$ 4\|\nabla_x \Theta\|_2^2-\frac{\alpha d(\alpha d+4)}{4(\alpha+2)}\|\Theta\|_{\alpha+2}^{\alpha+2}\geq 0.$$
Then
\begin{align}
\mK(\Theta)\geq\bg(1-\frac{8}{\alpha d+4}\bg) \|\nabla_x \Theta\|_2^2.
\end{align}
Since $\alpha>4/d$, we conclude that $1-\frac{8}{\alpha d+4}>0$ and the claim follows.

\item Next we consider the case
\begin{align}\label{hyp2}
4\|\nabla_x \Theta\|_2^2-\frac{\alpha d(\alpha d+4)}{4(\alpha+2)}\|\Theta\|_{\alpha+2}^{\alpha+2}< 0.
\end{align}
Define $f(t):=\mH(\Theta^t)$, where $\Theta^t$ is defined by \eqref{def of scaling op}. Direct calculation yields
\begin{align}
(\mK(\Theta^t))'&=(t f'(t))'=-2f'(t)+t\bg(4\|\nabla_x \Theta\|_2^2-t^{\frac{\alpha d}{2}-1}\frac{\alpha d(\alpha d+4)}{4(\alpha+2)}\|\Theta\|_{\alpha+2}^{\alpha+2}\bg)=-2f'(t)+th(t).
\end{align}
Using \eqref{hyp2} we know that $h(t)<0$ for all $t\in[1,\infty)$, thus
\begin{align}
(\mK(\Theta^t))'\leq -2f'(t)\quad\text{for all $t\in[1,\infty)$}.\label{5.7}
\end{align}
Since $\mK(\Theta)>0$, by Lemma \ref{monotone lemma} we can find some $t_0\in(1,\infty)$ such that $\mK(\Theta^{t_0})=0$. Moreover, since $\mM(\Theta)<\mM(u)=c$, by Lemma \ref{lem tilde mc} and \ref{monotone lemma} we have $\mI(\Theta^{t_0})\geq m_c$. Finally, integrating \eqref{5.7} yields
\begin{align}
\mK(\Theta)&\geq 2(\mH(\Theta^{t_0})-\mH(\Theta))=2\bg(\mI(\Theta^{t_0})-\mI(\Theta)-\frac{2}{\alpha d}\mK(\Theta)\bg),
\end{align}
which combining with $\mI(\Theta)\leq m_c-\frac{\nu}{2}$ implies
\begin{align}
\mK(\Theta)&\geq 2\bg(1+\frac{4}{\alpha d}\bg)^{-1}(\mI(\Theta^{t_0})-\mI(\Theta))
\geq \nu\bg(1+\frac{4}{\alpha d}\bg)^{-1}=
\frac{\nu}{m_c}\bg(1+\frac{4}{\alpha d}\bg)^{-1}m_c.\label{5.10}
\end{align}
But
\begin{align}
m_c>\mI(\Theta)=\frac{1}{2}\|\pt_y \Theta\|_2^2+\bg(\frac{1}{2}-\frac{2}{\alpha d}\bg)\|\nabla_x \Theta\|_2^2\geq
\bg(\frac{1}{2}-\frac{2}{\alpha d}\bg)\|\nabla_x \Theta\|_2^2.\label{5.11}
\end{align}
The desired claim then follows from \eqref{5.10} and \eqref{5.11}.
\end{itemize}
\end{proof}

We next derive the IMDM-identity in the waveguide setting.

\begin{lemma}[IMDM-identity]\label{lem morawetz identity}
Define
\begin{align}
M(t)&:=2\int|u(t,x_b,y_b)|^2\psi_R(x_a-x_b)(x_a-x_b)\nonumber\\
&\times\mathrm{Im}(\bar{u}(t,x_a,y_a)\nabla_x u(t,x_a,y_a))\,d(x_a,y_a)d(x_b,y_b),
\end{align}
where $\psi\in C_c^\infty(\R^d,[0,\infty))$ is the Morawetz potential defined in \cite{DodsonMurphyNonRadial} satisfying \cite[(4.1)-(4.2)]{DodsonMurphyNonRadial} and $\psi_R(x)=\psi(x/R)$. Then
\begin{align}
\frac{d}{dt}M(t)&=-4\int\pt_{x_{j}}(\mathrm{Im}(\bar{u}(t,x_b,y_b)\pt_{x_{j}} u(t,x_b,y_b)))\nonumber\\
&\times \psi_R(x_a-x_b)(x_{a,k}-x_{b,k})(\mathrm{Im}(\bar{u}(t,x_a,y_a)\pt_{x_{k}} u(t,x_a,y_a)))\,d(x_a,y_a)d(x_b,y_b)\label{5.13}
\\
\nonumber\\
&-4\int\pt_{x_{j}}|u(t,x_b,y_b)|\psi_R(x_a-x_b)(x_{a,j}-x_{b,j})\nonumber\\
&\times \pt_{x_{k}}(\mathrm{Re}(\pt_{x_{k}}\bar{u}(t,x_a,y_a)\pt_{x_{j}} u(t,x_a,y_a)))\,d(x_a,y_a)d(x_b,y_b)\label{5.14}\\
\nonumber\\
&+\int |u(t,x_b,y_b)|^2\psi_R(x_a-x_b)(x_a-x_b)\cdot\nabla_{x}\Delta_x(|u(t,x_a,y_a)|^2)\,d(x_a,y_a)d(x_b,y_b)\label{5.15}\\
\nonumber\\
&+\frac{2\alpha}{\alpha+2}
\int |u(t,x_b,y_b)|^2\psi_R(x_a-x_b)(x_a-x_b)\cdot\nabla_x (|u(t,x_a,y_a)|^{\alpha+2})\,d(x_a,y_a)d(x_b,y_b)\label{5.16}.
\end{align}
\end{lemma}

\begin{proof}
Since $u$ is a solution of \eqref{nls}, we know that
$$ \pt_t u=i(\Delta_{x}+|u|^\alpha)u+i\pt^2_y u.$$
An integration by parts corresponding to the part $i(\Delta_{x}+|u|^\alpha)u$ already yields the sum of \eqref{5.13} to \eqref{5.16}, see \cite[Lem. 4.3]{InteractionCombined}. It is left to show that the part corresponding to $i\pt^2_y u$ does not contribute. More precisely, we will show that the following sum
\begin{align}
-&\,4\int\mathrm{Im}(\bar{u}(t,x_b,y_b)\pt^2_y u(t,x_b,y_b))\nonumber\\
&\,\psi_R(x_a-x_b)(x_a-x_b)\cdot\mathrm{Im}(\bar{u}(t,x_a,y_a)\nabla_x u(t,x_a,y_a))\,d(x_a,y_a)d(x_b,y_b)
\label{5.17}\\
\nonumber\\
-&\,2\int|u(t,x_b,y_b)|^2\psi_R(x_a-x_b)(x_a-x_b)\cdot\mathrm{Re}(\pt_y^2\bar{u}(t,x_a,y_a)\nabla_x u(t,x_a,y_a))\,d(x_a,y_a)d(x_b,y_b)
\label{5.18}\\
\nonumber\\
+&\,2\int|u(t,x_b,y_b)|^2\psi_R(x_a-x_b)(x_a-x_b)\cdot\mathrm{Re}(\bar{u}(t,x_a,y_a)\pt_y^2\nabla_x u(t,x_a,y_a))\,d(x_a,y_a)d(x_b,y_b).\label{5.19}
\end{align}
is equal to zero. Since $\bar{u}\pt_y^2 u=\pt_y(\bar{u}\pt_y u)-|\pt_y u|^2$, $\eqref{5.17}=0$ follows from integration by parts, the periodic boundary condition of $u$ along the $y$-direction and the fact that $|\pt_y u|^2$ is real-valued and $\psi$ is independent of $y$. Another application of integration by parts yields $\eqref{5.18}+\eqref{5.19}=0$. This completes the proof.
\end{proof}

We point out that \eqref{5.13}-\eqref{5.16} coincide with \cite[(4.17)-(4.21)]{InteractionCombined}. Thus Lemma \ref{lem morawetz identity} enables us to directly integrate \cite[(4.23)]{InteractionCombined} over $\T_{y_a}\times\T_{y_b}$ to deduce the following IMDM-inequality, where we also need to replace \cite[Lem. 4.2]{InteractionCombined} to Lemma \ref{lem coercive}. Since the adaptation is straightforward, we omit the details of the proof.

\begin{lemma}[IMDM-inequality]\label{lem inter morawetz}
For any $\vare>0$ there exist $T_0=T_0(\vare)\gg 1$, $J=J(\vare)\gg 1$, $R_0=R_0(\vare,u_0)\gg 1$ and $\eta=\eta(\vare)\ll 1$ such that for any $a\in\R$ we have
\begin{align}
&\,\frac{1}{JT_0}\int_a^{a+T_0}\int_{R_0}^{R_0e^J}\frac{1}{R^d}\int_{(\R_{x_a}^d\times\T_{y_a})\times (\R_{x_b}^d\times\T_{y_b})\times\R^d_z}
\nonumber\\
&\qquad |\chi_R(x_b-z)u(t,x_b,y_b)|^2|\nabla_x(\chi_R(x_a-z)u^\xi(t,x_a,y_a))|^2d(x_a,y_a)d(x_b,y_b)dz\frac{dR}{R}dt\lesssim \vare.\label{5.12}
\end{align}
\end{lemma}

\section{Proof of Theorem \ref{main result}}

Having all the preliminaries we are in a position to prove Theorem \ref{main result}.

\begin{proof}[Proof of Theorem \ref{main result}]
Suppose that we are given some $\sigma>0$ which will be chosen sufficiently small later. By Lemma \ref{lem scattering criterion}, the claim follows as long as we can show that there exist $\vare=\vare(\sigma,A)$ sufficiently small and $T_0=T_0(\sigma,\vare,A)$ sufficiently large such that if for all $a\in\R$ there exists $T\in(a,a+T_0)$ such that $[T-\vare^{-\sigma},T]\subset(a,a+T_0)$ and
\begin{align}
\|u\|_{L_t^\ba L_x^\br H_y^s(T-\vare^{-\sigma},T)}\lesssim \vare^\mu
\end{align}
for some $\mu>0$. Let the notations of Lemma \ref{lem inter morawetz} be retained. Then by \eqref{5.12} we know that there exists some $R_1\in[R_0,R_0 e^J]$ such that
\begin{align}
&\,\frac{1}{T_0}\int_a^{a+T_0}\frac{1}{R_1^d}
\|\chi_{R_1}(\cdot-z)u(t)\|_{L_{x,y}^2}^2\|\nabla_x(\chi_{R_1}(\cdot-z)u^\xi(t))\|_{L_{x,y}^2}^2dzdt\lesssim \vare.
\end{align}
Writing $z=\frac{R_1}{4}(w+\theta)$ with $w\in \Z^d$ and $\theta\in[0,1]^d$ and using mean value theorem we infer that there exists $\theta_0:\Z^d\to [0,1]^d$ such that
\begin{align}
&\,\frac{1}{T_0}\int_a^{a+T_0}\sum_{w\in\Z^d}
\|\chi_{R_1}(\cdot-\frac{R_1}{4}(w+\theta_0))u(t)\|_{L_{x,y}^2}^2\|\nabla_x(\chi_{R_1}(\cdot-\frac{R_1}{4}(w+\theta_0))u^\xi(t))\|_{L_{x,y}^2}^2 dt\lesssim \vare.
\end{align}
By diving $[a+T_0/2,a+3T_0/4]$ into subintervals of length $\vare^{-\sigma}$ we know that there exists some $t_0\in[a+T_0/2,a+3T_0/4]$ such that $[t_0-\vare^{-\sigma},t_0]\subset[a,a+T_0]$ and
\begin{align}
&\,\int_{t_0-\vare^{-\sigma}}^{t_0}\sum_{w\in\Z^d}
\|\chi_{R_1}(\cdot-\frac{R_1}{4}(w+\theta_0))u(t)\|_{L_{x,y}^2}^2\|\nabla_x(\chi_{R_1}(\cdot-\frac{R_1}{4}(w+\theta_0))u^\xi(t))\|_{L_{x,y}^2}^2 dt\lesssim \vare^{1-\sigma}.
\end{align}
Next, from the modified Gagliardo-Nirenberg inequality on $\R^d$ (see \cite[Lem. 2.1]{DodsonMurphyNonRadial})
\begin{align}
\|u\|^2_{L_x^{\frac{2d}{d-1}}}\lesssim\|u\|_{L_x^{2}}\|\nabla_x u^\xi\|_{L_x^{2}}\label{6.5}
\end{align}
we deduce by combining H\"older and Minkowski that
\begin{align}
\|u\|^4_{L_x^{\frac{2d}{d-1}}L_y^2}\lesssim \|u\|^4_{L_y^2L_x^{\frac{2d}{d-1}}}
\lesssim \|u\|^2_{L_{x,y}^{2}}\|\nabla_x u^\xi\|^2_{L_{x,y}^{2}},
\end{align}
which in turn implies
\begin{align}
\int_{t_0-\vare^{-\sigma}}^{t_0}\sum_{w\in\Z^d}\|\chi_{R_1}(\cdot-\frac{R_1}{4}(w+\theta_0))u(t)\|^4_{L_x^{\frac{2d}{d-1}}L_y^2}\,dt
\lesssim \vare^{1-\sigma}.\label{6.7}
\end{align}
On the other hand, using H\"older, Cauchy-Schwarz and the embedding $H_x^1\hookrightarrow L_x^{\frac{2d}{d-2}}$ we infer that
\begin{align}
&\,\sum_{w\in\Z^d}\|\chi_{R_1}(\cdot-\frac{R_1}{4}(w+\theta_0))u(t)\|^2_{L_x^{\frac{2d}{d-1}}L_y^2}\nonumber\\
\lesssim&\,
\sum_{w\in\Z^d}\|\chi_{R_1}(\cdot-\frac{R_1}{4}(w+\theta_0))u(t)\|_{L_{x,y}^2}
\|\chi_{R_1}(\cdot-\frac{R_1}{4}(w+\theta_0))u(t)\|_{L_x^{\frac{2d}{d-2}}L_{y}^2}\nonumber\\
\lesssim&\,\sum_{w\in\Z^d}\|\chi_{R_1}(\cdot-\frac{R_1}{4}(w+\theta_0))u(t)\|_{L_{x,y}^2}
\|\chi_{R_1}(\cdot-\frac{R_1}{4}(w+\theta_0))u(t)\|_{H_{x}^{1}L_{y}^2}\nonumber\\
\lesssim&\, \|u(t)\|_{H_{x,y}^1}^2+O(R_1^{-2}\eta^{-2})\|u(t)\|_{L_x^2}^2\lesssim 1,
\end{align}
which in turn implies
\begin{align}
\int_{t_0-\vare^{-\sigma}}^{t_0}\sum_{w\in\Z^d}\|\chi_{R_1}(\cdot-\frac{R_1}{4}(w+\theta_0))u(t)\|^2_{L_x^{\frac{2d}{d-1}}L_y^2}\,dt
\lesssim \vare^{-\sigma}.\label{6.8}
\end{align}
Interpolating \eqref{6.7} and \eqref{6.8} yields
\begin{align}
&\,\|u\|^{\frac{2d}{d-1}}_{L_{t,x}^{\frac{2d}{d-1}}L_y^2(t_0-\vare^{-\sigma},t_0)}\lesssim
\int_{t_0-\vare^{-\sigma}}^{t_0}\sum_{w\in\Z^d}\|\chi_{R_1}(\cdot-\frac{R_1}{4}(w+\theta_0))u(t)\|^{\frac{2d}{d-1}}_{L_x^{\frac{2d}{d-1}}L_y^2}\,dt\nonumber\\
\lesssim&\,
\int_{t_0-\vare^{-\sigma}}^{t_0}\bg(\sum_{w\in\Z^d}
\|\chi_{R_1}(\cdot-\frac{R_1}{4}(w+\theta_0))u(t)\|^{4}_{L_x^{\frac{2d}{d-1}}L_y^2}\bg)^{\frac{1}{d-1}}
\nonumber\\
\times&\,\bg(\sum_{w\in\Z^d}
\|\chi_{R_1}(\cdot-\frac{R_1}{4}(w+\theta_0))u(t)\|^{2}_{L_x^{\frac{2d}{d-1}}L_y^2}\bg)^{\frac{d-2}{d-1}}\,dt
\nonumber\\
\lesssim&\,
\bg(\int_{t_0-\vare^{-\sigma}}^{t_0}\sum_{w\in\Z^d}
\|\chi_{R_1}(\cdot-\frac{R_1}{4}(w+\theta_0))u(t)\|^{4}_{L_x^{\frac{2d}{d-1}}L_y^2}\,dt\bg)^{\frac{1}{d-1}}
\nonumber\\
\times&\,\bg(\int_{t_0-\vare^{-\sigma}}^{t_0}\sum_{w\in\Z^d}
\|\chi_{R_1}(\cdot-\frac{R_1}{4}(w+\theta_0))u(t)\|^{2}_{L_x^{\frac{2d}{d-1}}L_y^2}\,dt\bg)^{\frac{d-2}{d-1}}
\lesssim \vare^{\frac{1}{d-1}-\sigma}.
\end{align}
Using $\br>\frac{\alpha(\alpha+1)d}{\alpha+2}$ from Lemma \ref{exotic strichartz} and $\alpha>4/d$ one easily verifies that $\br>\frac{2d}{d-1}$. For $\theta\in(0,1)$, we define $(\ba_\theta,\br_\theta,s_\theta,\tilde{s}_\theta)$ via
\begin{gather}
\frac{1}{\br}=\frac{\theta}{(2d/(d-1))}+\frac{1-\theta}{\br_\theta},\quad
\frac{1}{\ba}=\frac{\theta}{(2d/(d-1))}+\frac{1-\theta}{\ba_\theta},\\
s_{\alpha}=\frac{2-d}{2d}\theta+(1-\theta)s_\theta,\quad s=(1-\theta)\tilde{s}_\theta.
\end{gather}
Then $(\ba_\theta,\br_\theta)$ is $H^{s_\theta}$-admissible when $\theta$ is close to zero. We also define $\tilde{\br}_\theta$ such that $(\ba_\theta,\tilde{\br}_\theta)$ is $L^2$-admissible. Moreover,
$$ s_\theta+\tilde{s}_\theta=(1-\theta)^{-1}\bg(s_\alpha+s+\frac{d-2}{2d}\theta\bg)\to s_\alpha+s<1$$
as $\theta\to 0$. Thus by choosing $\theta\ll 1$ we have $s_\theta<1-\tilde{s}_\theta$. By continuity we also know that $\tilde{s}_\theta\in(\frac12,s_\alpha)$ for $\theta\ll 1$. Thus fix some $\theta$ sufficiently close to zero. Then by interpolation and Lemma \ref{local control} we infer that
\begin{align}
\|u\|_{L_t^\ba L_x^\br H_y^s(t_0-\vare^{-\sigma},t_0)}&\lesssim
\|u\|_{L_{t,x}^{\frac{2d}{d-1}} L_y^2(t_0-\vare^{-\sigma},t_0)}^\theta
\|u\|_{L_t^{\ba_\theta} L_x^{\br_\theta} H_y^{\tilde{s}_\theta}(t_0-\vare^{-\sigma},t_0)}^{1-\theta}\nonumber\\
&\lesssim
\|u\|_{L_{t,x}^{\frac{2d}{d-1}} L_y^2(t_0-\vare^{-\sigma},t_0)}^\theta
\|u\|_{L_t^{\ba_\theta} L_x^{\br_\theta} H_y^{\tilde{s}_\theta}(t_0-\vare^{-\sigma},t_0)}^{1-\theta}
\nonumber\\
&\lesssim
\vare^{\theta(\frac{1}{2d}-\frac{d-1}{2d}\sigma)}
\|u\|_{L_t^{\ba_\theta} W_x^{1-\tilde{s}_\theta,\tilde{\br}_\theta} H_y^{\tilde{s}_\theta}(t_0-\vare^{-\sigma},t_0)}^{1-\theta}
\lesssim
\vare^{\theta(\frac{1}{2d}-\frac{d-1}{2d}\sigma)}
\vare^{-\frac{\sigma(1-\theta)}{\ba_\theta}}.
\end{align}
The desired claim follows by choosing $\sigma$ small.
\end{proof}

\subsubsection*{Acknowledgements}
The author acknowledges the funding by Deutsche Forschungsgemeinschaft (DFG) through the Priority Programme SPP-1886 (No. NE 21382-1).

%%%%%%%%%%%%%%%%%%%%%%%%%%%%%

%\addcontentsline{toc}{section}{References}
%\bibliography{biharmonic}
%\bibliographystyle{acm}

\end{document}